\documentclass{amsart}

\usepackage{}

\newtheorem{theorem}{Theorem}[section]
\newtheorem{lemma}[theorem]{Lemma}

\theoremstyle{definition}

\theoremstyle{remark}

\numberwithin{equation}{section}

\begin{document}

\title[Jordan Nilpotent Group Rings]{Jordan Nilpotent Group Rings of index $4$}


\author{Meena Sahai}
\address{Department of Mathematics and Astronomy, University of Lucknow, Lucknow 226007, India.}
\curraddr{}
\email{sahai\_m@lkouniv.ac.in}
\thanks{Corresponding Author}

\author{Sachin Singh}
\address{Department of Mathematics and Astronomy, University of Lucknow, Lucknow 226007, India.}
\curraddr{}
\email{sachinsinghndps@gmail.com}
\thanks{}


\date{}

\begin{abstract}
Let $RG$ be the group ring of an arbitrary group $G$ over an associative non-commutative ring $R$  with identity.  In this paper, we have obtained the necessary and sufficient conditions under which $RG$ is Jordan nilpotent of index $4$.
\end{abstract}

\maketitle

\section{Introduction}

Any associative ring $S$ can be regarded as a Jordan ring under the circle operation $\alpha \circ \beta = \alpha \beta + \beta \alpha$, for all $\alpha,\beta \in S$ and for any integer $n\geq 2$, $\alpha^{\circ n} = \alpha^{\circ (n-1)}\circ \alpha$.  The circle operation or the Jordan product is such that for all $\alpha,\beta \in S$
\begin{enumerate}
\item  $\alpha \circ \beta = \beta \circ \alpha$;
\item $(\alpha^{\circ 2}\circ \beta)\circ \alpha = \alpha^{\circ 2}\circ(\beta\circ \alpha)$.
\end{enumerate}
 The ring $S$ is Jordan nilpotent of index $n$ if $(\dots((\alpha_{1}\circ \alpha_{2})\circ \alpha_{3})\dots)\circ \alpha_{n}=0$ for all $\alpha_1, \alpha_2, \alpha_3,\dots ,\alpha_n\in S$.
 
The Lie product on $S$ is defined as $[\alpha, \beta] = \alpha \beta - \beta \alpha$,  for all $\alpha,\beta \in S$ which satisfies the twin identities:
\begin{enumerate}
\item  $[\alpha,\alpha]=0$;
\item $[[\alpha, \beta],\gamma] + [[\beta,\gamma],\alpha]+[[\gamma,\alpha],\beta]=0$
\end{enumerate} 
for all $\alpha,\beta,\gamma \in S$. With the operations of addition and Lie product, $S$ becomes a Lie ring. $S$ is Lie nilpotent of index $n$ if $[\dots[[\alpha_{1}, \alpha_{2}], \alpha_{3}],\dots ,\alpha_{n}]=0$ for all $\alpha_1, \alpha_2, \alpha_3,\dots ,\alpha_n\in S$.  It is important to know the conditions under which $S$ is Lie nilpotent or Jordan nilpotent. In particular, Lie nilpotent and Lie solvable group rings over commutative rings have been studied well, see \cite{PPS}.   However, not much is known about Jordan nilpotent group rings.  Goodaire and Milies \cite{GM} have classified  Jordan nilpotent group rings of index $3$ over commutative rings. 
\begin{theorem}\label{t2}\cite[Theorem 2.4]{GM}
The group ring $RG$ of a group $G$ over a commutative ring $R$ is Jordan nilpotent of index $3$  if, and only if, one of the following statements holds:
\begin{itemize}
\item[1)] $Char(R)=4$ and $G$ is abelian
\item[2)] $Char(R)=2$ and  $G'\simeq C_2$.
\end{itemize}
\end{theorem}
They also discussed the  Jordan nilpotency of the symmetric elements of the group ring with respect to an involution of the underlying group.  In \cite{CHR}, Jordan nilpotent group rings of index $2$  and Jordan nilpotent group rings of index $3$ over non-commutative rings are characterized. They have also studied  Jordan nilpotent group rings of index $4$ over commutative rings.
\begin{theorem}\label{t1}\cite[Theorem 2.1]{CHR}
The group ring $RG$ of a group $G$ over a ring $R$ is Jordan nilpotent of index $2$ if, and only if, $R$ is commutative, $Char(R)=2$ and $G$ is abelian.
\end{theorem}
\begin{theorem}\label{t3}\cite[Theorem 3.1]{CHR}
The group ring $RG$ of a group $G$ over a non-commutative ring $R$ is Jordan nilpotent of index $3$  if, and only if, $G$ is abelian, $R$ is Jordan nilpotent of index $3$ with $Char(R)=2$ or $4$.
\end{theorem}
\begin{theorem}\label{t4}\cite[Theorem 4.1]{CHR}
The group ring $RG$ of a group $G$ over a commutative ring $R$ is Jordan nilpotent of index $4$  if, and only if, one of the following statements holds:
\begin{itemize}
\item [1)] $Char(R)=8$, $G$ is abelian,
\item [2)] $Char(R)=4$, $G'\simeq C_2$,
\item [3)] $Char(R)=2$, $G'\simeq C_2 \times C_2$ and $G'\subseteq Z(G)$.
\end{itemize}
\end{theorem}
  In this paper, we give a complete characterization of  Jordan nilpotent group rings of index $4$ over non-commutative rings. Working with non-commutative rings is much more challenging.  Levin and Rosenberger \cite{LR} have studied Lie metabelian group rings over both commutative and non-commutative rings. Sahai and Ansari \cite{SA} have given characterization of Lie centrally metabelian group rings over non-commutative rings and Lie centrally metabelian group algebras are classified in \cite{KS, SSr, ShS}. In this context, it is important to study Jordan nilpotent group rings over non-commutative rings.

 Our notations are standard.  Throughout  $G$ is a group, $R$ is an associative ring with identity $1$, and $RG$ is the group ring of $G$ over $R$. $G'$ and $Z(G)$ denote the derived subgroup and the centre of $G$, respectively. $C_n$ is cyclic group of order $n$, whereas $Char(R)$ is characteristic of  $R$. Also $(x,y)=x^{-1}y^{-1}xy$ is group commutator of $x, y \in G$ and $x^y=y^{-1}xy$. The following identities are used frequently without mentioning:
$$(xy,z)= (x,z)^y(y,z),~~~~~~~~~(x,yz)=(x,z)(x,y)^z,$$
  
$$x\circ y=yx((x,y)+1),~~~~~x^{-1}y^{-1} \circ x =((x,y)+1)y^{-1}$$
and $$y^{-1}x \circ y = x((x,y)+1),$$
for all $x,y,z\in G.$

Now, we present some preliminary results given in \cite{CHR, GM}, which are needed for our work and in the next section, we give a complete characterization of  a Jordan nilpotent group ring of index $4$ over  a non-commutative ring.
\begin{lemma}\label{l1}\cite{CHR}
If a ring $R$ is Jordan nilpotent of index $n$, then  $Char(R)$ divides $2^{n-1}$.
\end{lemma}
 \begin{proof} This is straightforward as $2^{n-1} = 1^{\circ n}=0$.
 \end{proof}
\begin{lemma}\label{l2}\cite[Proposition 2.1]{CHR}
If a ring  $R$ has characteristic $2^{n-1}$ and $G$ is a group such that  $RG$ is Jordan nilpotent of index $n\geq2$, then $G$ is abelian.
\end{lemma}

\begin{lemma}\label{l3}\cite[Section 2]{GM}
If a  group  $G$ is such that $G^2\subseteq Z(G)$, then $G'\subseteq Z(G)$ and $G'^2=1$.
\end{lemma}

\begin{lemma}\label{l4}\cite[Lemma 4]{JS}
	Let $G$ be a non-abelian group such that, for all $x,y\in G$ with $(x,y)\neq1$, we have $(y,z)\in \langle(x,y)\rangle$ for all $z\in G$. Then $G'$ is cyclic.
\end{lemma}

\section{Main Result}
In this section, we assume that $R$ is a non-commutative ring. The aim is to identify the conditions under which the group ring $RG$ of a group $G$ over  $R$ is Jordan nilpotent of index $4$ and thereby complete the work done in \cite{CHR}. We begin by presenting certain identities whose significance will become apparent later in this paper.

\begin{lemma}\label{l6}
Let $R$ be a ring. Then for all $\alpha, \beta, \gamma \in R$, $\alpha \beta \circ \gamma = \alpha(\beta\circ \gamma)+(\gamma\circ \alpha)\beta-2\alpha \gamma \beta$. 
\end{lemma}
\begin{proof}
\begin{align*}	
\alpha \beta\circ \gamma&=\alpha \beta \gamma + \gamma \alpha \beta + 2 \alpha \gamma\beta - 2\alpha \gamma \beta\\
&=\alpha(\beta \gamma + \gamma\beta)+(\gamma\alpha+\alpha \gamma)\beta-2\alpha \gamma\beta\\
&=\alpha(\beta\circ \gamma)+(\alpha\circ \gamma)\beta-2 \alpha \gamma\beta.
\end{align*}
\end{proof}
\begin{lemma}\label{l7}
Let $G$ be a group and $R$ be a ring. Then for all $x,y\in G$ and $\alpha, \beta\in R$,  
$$\alpha x\circ \beta y=(\alpha \circ \beta)yx+ \alpha \beta yx((x,y)-1).$$
\end{lemma}
\begin{proof}
\begin{align*}		
\alpha x \circ \beta y&= \alpha \beta xy + \beta \alpha yx \\ 
    	&=\alpha \beta yx((x,y)-1) + \alpha \beta yx + \beta \alpha yx \\ 
		&=\alpha \beta yx((x,y)-1) + (\alpha \circ \beta)yx. 		
\end{align*}
\end{proof}
\begin{theorem}\label{t5}
Let $RG$ be the group ring of a group $G$ over  a non-commutative  ring $R$. Then  $RG$ is Jordan nilpotent of index $4$ if, and only if, one of the following conditions holds:
\begin{enumerate}
\item [1)] $G$ is abelian, $R$ is Jordan nilpotent of index $4$ with $Char(R)=2$ or $4$ or $8$; 
\item [2)] $Char(R)=4$,  $2(R\circ R)=0$, $(R\circ R)\circ R=0$, $(R\circ R)(R\circ R)=0$, and $G'\simeq C_2$.
\item [3)] m$Char(R)=2$,  $(R\circ R)\circ R=0$, $(R\circ R)(R\circ R)=0$ and $G'\simeq C_2$.
\end{enumerate}
\end{theorem}
\begin{proof}
 First we prove that the given conditions are necessary. Suppose that $RG$ is Jordan nilpotent of index $4$. By Lemma \ref{l1}, $Char(R) =2$ or $4$ or $8$. Obviously,  $R$ is Jordan nilpotent of index $4$. Let $G$ be non-abelian.   By Lemma \ref{l2}, $Char(R)=2$ or $4$.

Suppose $Char(R)=4$ and $x,y\in G$ such that $s=(x,y)\neq 1$. Then for $\alpha, \beta, \gamma, \delta \in R$,		
$$0= ( (\alpha x \circ y )\circ \beta )\circ 1=2 (\alpha \circ \beta ) (x \circ y )=2(\alpha \circ \beta)yx(s+1).$$
which implies that $2(R\circ R)=0$. Further,
$$0= ( (\alpha x \circ y )\circ \beta )\circ \gamma = ((\alpha \circ \beta )\circ \gamma ) (x \circ y ) = ((\alpha \circ \beta)\circ \gamma)yx(s+1)$$
 and hence $(R\circ R)\circ R=0$. Now for any $g,h \in G$,
$$0= ( (g\circ h )\circ h )\circ 1= 2 (gh^2+2ghg+h^2g )$$
 yields that $gh^2=h^2g$. Therefore, $G^2\subseteq Z(G)$ and by Lemma \ref{l3}, $G'\subseteq Z(G)$ and $G'^2=1$.  If $z \in G$, then
\begin{align*}		
0 & =\big( (x^{-1}y^{-1} \circ x) \circ zy\big)\circ 1\\
   &= 2y^{-1}(s+1)\circ zy\\
   &= 2z((z,y)+1)(s+1).		
\end{align*}
 Hence, $(y,z) \in \langle s \rangle$ and $G'\simeq C_2$ by Lemma \ref{l4}.
 
Consider,  $( (\alpha x\circ \beta )\circ \gamma y )\circ \delta =0$. Using Lemmas \ref{l6} and \ref{l7} and the fact that as $(R\circ R)\circ R=0$ and $2(R\circ R)=0$, we get
 \begin{align*}
0&=( (\alpha x\circ \beta )\circ \gamma y )\circ \delta\\  
&= ( (\alpha \circ \beta)x\circ \gamma y )\circ \delta\\ 
 &=  \big( ( (\alpha \circ \beta)\circ \gamma \big) yx +(\alpha \circ \beta)\gamma yx(s-1)\big)\circ \delta\\ 
 &= \big( (\alpha \circ \beta)\gamma \circ \delta \big)yx(s-1) \\
 &= \big( (\alpha \circ \beta)(\gamma \circ \delta)+((\alpha \circ \beta)\circ \delta)\gamma - 2(\alpha \circ \beta)\delta\gamma \big)yx(s-1)\\
 &= (\alpha \circ \beta)(\gamma \circ \delta)yx(s-1).
\end{align*}

 Thus,  $(R\circ R)(R\circ R)=0.$
 
Now, let $Char(R)=2$. Choosing  $x,y\in G$ such that $s=(x,y)\neq 1$ and $\alpha, \beta, \gamma, \delta \in R$ as in the previous case, $( (\alpha x\circ y)\circ \beta )\circ \gamma = 0$ yields  that  $(R\circ R)\circ R=0.$ Further, for $g, h \in G$,
 $$0=((g\circ h)\circ h)\circ h=gh^3+hgh^2+h^2gh+h^3g$$
yields that $gh^3=hgh^2$ or $h^2gh$ or $h^3g$. In the first of these conditions $gh=hg$ and in the second condition $gh^2=h^2g$.  Also, if  $gh^3=h^3g$, then $hgh^2=h^2gh$, which  again gives $gh=hg$. Therefore, we conclude that  $gh^2=h^2g$. So,  $G^2\subseteq Z(G)$ and by Lemma \ref{l3},  $G'\subseteq Z(G)$ and $G'^2=1$.
We also have
\begin{align*}		
0 &= \big( (x^{-1}y^{-1} \circ x) \circ zy\big)\circ z^{-1}x\\
   &= \big(y^{-1}(s+1)\circ zy\big)\circ z^{-1}x\\
   &= z((z,y)+1)(s+1)\circ z^{-1}x\\
   &= (z\circ z^{-1}x)((z,y)+1)(s+1)\\
   &= x((x,z))+1)((z,y)+1)(s+1),		
\end{align*}
for all $ z\in G$.
Hence, we have  the following possibilities :
\begin{enumerate}
\item $(y,z) \in \langle s \rangle$;
\item $(x,z) \in \langle s \rangle$;
\item $(y,z), (x,z) \notin \langle s \rangle$ and $(x,z) \in \langle (y,z) \rangle$.
\end{enumerate}
In all these cases, $G'\simeq C_2$ by Lemma \ref{l4}.

Next,  $( (\alpha x\circ \beta )\circ \gamma y )\circ \delta = 0$ implies that $(R\circ R)(R\circ R)=0$ as we have seen in the previous case.

Conversely, if $G$ is abelian and $R$ is Jordan nilpotent of index $4$, then
$$((RG \circ RG) \circ RG) \circ RG \subseteq (((R \circ R) \circ R) \circ R)G=0.$$
 
Let $Char(R)=4$,  $2(R\circ R)=0$, $(R\circ R)\circ R=0$, $(R\circ R)(R\circ R)=0$, and $G'\simeq C_2 =\langle s \rangle$. It is easy to see that $(s-1)^2= -2(s-1)$ and $(s-1)^3 = 0$ and  $G' \subseteq Z(G)$.  Let $x,y,z,t\in G$, $(x,y)=s_1$, $(yx,z)=s_2$ and $(zyx,t)=s_3$. Then for $\alpha,\beta, \gamma, \delta \in R$, we have
$$\big((\alpha\circ \beta)\gamma \circ \delta \big)=(\alpha \circ \beta)(\gamma \circ \delta)+\big((\alpha\circ \beta)\circ \delta \big)\gamma - 2(\alpha \circ \beta)\delta \gamma=0$$
and
$$\alpha x\circ \beta y=(\alpha \circ \beta)yx + \alpha \beta yx(s_1-1),$$
\textbf{Case (1)} $s_1=1.$ \\
\begin{align*}
(\alpha x\circ \beta y)\circ \gamma z &=(\alpha \circ \beta)yx\circ \gamma z\\
&=((\alpha \circ \beta)\circ \gamma)zyx+(\alpha\circ \beta)\gamma zyx(s_2-1)\\
&=(\alpha \circ \beta)\gamma zyx(s_2-1).
\end{align*}
If $s_2=1$, then $(\alpha x\circ \beta y)\circ \gamma z=0$. Otherwise, $s_2=s$ and $(\alpha x \circ \beta y)\circ \gamma z = (\alpha \circ \beta)\gamma zyx(s-1)$. Now,
\begin{align*}
\big(( \alpha x\circ \beta y)\circ \gamma z\big)\circ \delta t &=\big( (\alpha \circ \beta)\gamma zyx\circ \delta t\big)(s-1)\\
&=\Big( \big((\alpha \circ \beta)\gamma \circ \delta \big)tzyx+(\alpha \circ \beta)\gamma \delta tzyx(s_3-1)\Big)(s-1)\\
&=(\alpha \circ \beta)\gamma \delta tzyx(s_3-1)(s-1).
\end{align*}
If $s_3= 1$, then $\big((\alpha x\circ \beta y)\circ \gamma z\big)\circ \delta t=0$.  On the other hand, if $s_3=s$, then again
$$\big((\alpha x\circ \beta y)\circ \gamma z\big)\circ \delta t= -2(\alpha \circ \beta)\gamma\delta tzyx(s-1)=0.$$
 \textbf{Case (2)} $s_1=s.$ 
$$ \alpha x\circ\beta y=(\alpha\circ \beta)yx + \alpha \beta yx(s-1)$$
and 
 \begin{align*}
(\alpha x\circ \beta y) \circ \gamma z&=\big( (\alpha \circ \beta)yx+ \alpha \beta yx(s-1)\big)\circ \gamma z\\
&=\big( (\alpha \circ \beta)\circ \gamma \big)zyx+(\alpha \circ \beta)\gamma zyx(s_2-1)\\
&\quad+\big((\alpha \beta \circ \gamma)zyx+ \alpha \beta \gamma zyx(s_2-1)\big)(s-1).
\end{align*}
Therefore, if $s_2=1$, then $(\alpha x\circ \beta y) \circ \gamma z = (\alpha \beta \circ \gamma)zyx(s-1)$ and 
\begin{align*}
\big( (\alpha x\circ \beta y) \circ \gamma z\big)\circ \delta t  &=\big((\alpha \beta \circ \gamma)zyx \circ \delta t\big)(s-1)\\
&=\Big(\big((\alpha \beta \circ \gamma)\circ \delta \big)tzyx + (\alpha \beta \circ \gamma)\delta tzyx(s_3-1)\Big)(s-1).
\end{align*}
So, if $s_3=1$, then $\big( (\alpha x\circ \beta y) \circ \gamma z\big)\circ \delta t = 0$ and if $s_3=s$, then 
$$\big( (\alpha x\circ \beta y) \circ \gamma z\big)\circ \delta t = -2(\alpha \beta \circ \gamma)\delta tzyx(s-1)=0.$$
Now let $s_2=s$. Then
$$(\alpha x\circ \beta y) \circ \gamma z = \big( (\alpha \circ \beta)\gamma + (\alpha \beta \circ \gamma) -2\alpha \beta \gamma\big)zyx(s-1)$$
and
\begin{align*}
\big( (\alpha x\circ \beta y) \circ \gamma z \big) \circ \delta t	
&=\Big( \big( \big((\alpha \circ \beta)\gamma + (\alpha \beta \circ \gamma) -2\alpha \beta \gamma \big)zyx \big)\circ \delta t \Big)(s-1)\\
&=\Big( \big((\alpha \circ \beta)\gamma \circ \delta \big)tzyx+(\alpha \circ \beta)\gamma \delta tzyx(s_3-1)\\
&\quad +\big((\alpha \beta \circ \gamma)\circ \delta \big)tzyx+(\alpha \beta \circ \gamma)\delta tzyx(s_3-1)\\
&\quad -2(\alpha \beta \gamma \circ \delta)tzyx -2\alpha \beta \gamma \delta tzyx(s_3-1)\Big)(s-1)
\end{align*}
Hence $\big( (\alpha x\circ \beta y) \circ \gamma z \big) \circ \delta t =0$, if $s_3=1$. Whereas, if $s_3 =s$, then
\begin{align*}
\left( (\alpha x\circ \beta y) \circ \gamma z \right) \circ \delta t
&=2\big(-(\alpha \circ \beta)\gamma \delta -(\alpha \beta \circ \gamma)\delta + 2\alpha \beta \gamma \delta \big)tzyx(s-1)\\
&=0.
\end{align*}
Therefore, $RG$ is Jordan nilpotent of index $4$ in every possible case. 

It can be easily concluded from the above that if $Char(R)=2$,  $(R\circ R)\circ R=0$, $(R\circ R)(R\circ R)=0$ and $G'\simeq C_2$, then also $RG$ is Jordan nilotent of index $4$. This completes the proof of the theorem.
\end{proof}



\bibliographystyle{amsplain}

\end{document}